\documentclass[11pt]{amsart}

\usepackage{a4wide,amsmath,amssymb}
\usepackage[toc,page]{appendix}

\usepackage{url}
\usepackage{mathtools}
\usepackage{hyperref}

\usepackage{xcolor}

\makeatletter
\@namedef{subjclassname@2020}{\textup{2020} Mathematics Subject Classification}
\makeatother

\newtheorem{thm}{Theorem}[section]
\newtheorem{lemma}[thm]{Lemma}

\theoremstyle{remark}
\newtheorem{rem}[thm]{Remark}

\theoremstyle{definition}

\newtheoremstyle{Claim}{}{}{\itshape}{}{\itshape\bfseries}{:}{ }{#1}
\theoremstyle{Claim}

\newcommand{\R}{\mathbb{R}}

\newcommand{\X}{\mathcal{X}}

\newcommand{\eps}{\varepsilon}

\newcommand{\G}{\mathbb{G}}

\theoremstyle{plain}

\def\sideremark#1{\ifvmode\leavevmode\fi\vadjust{
\vbox to0pt{\hbox to 0pt{\hskip\hsize\hskip1em
\vbox{\hsize3cm\tiny\raggedright\pretolerance10000
\noindent #1\hfill}\hss}\vbox to8pt{\vfil}\vss}}}

\begin{document}

\title[A priori estimates for fully nonlinear subelliptic equations]{Interior a priori estimates for supersolutions of fully nonlinear subelliptic equations under geometric conditions}

\author{Alessandro Goffi}
\address{Dipartimento di Matematica ``Tullio Levi-Civita'', Universit\`a degli Studi di Padova, 
via Trieste 63, 35121 Padova (Italy)}
\curraddr{}
\email{alessandro.goffi@unipd.it}

\subjclass[2020]{35D40, 35B65, 35H20.}
\keywords{}
 \thanks{
 The author is member of the Gruppo Nazionale per l'Analisi Matematica, la Probabilit\`a e le loro Applicazioni (GNAMPA) of the Istituto Nazionale di Alta Matematica (INdAM). The author was partially supported by the INdAM-GNAMPA Project 2022 ``Propriet\`a quantitative e qualitative per EDP non lineari con termini di gradiente'' and by the King Abdullah University of Science and Technology (KAUST) project CRG2021-4674 ``Mean-Field Games: models, theory and computational aspects".}

\date{\today}

\date{\today}

\begin{abstract}
In this note, we prove interior a priori first- and second-order estimates for solutions of fully nonlinear degenerate elliptic inequalities structured over the vector fields of Carnot groups, under the main assumption that $u$ is semiconvex along the fields. These estimates for supersolutions are new even for linear subelliptic inequalities in nondivergence form, whereas in the nonlinear setting they do not require neither convexity nor concavity on the second derivatives. We complement the analysis exhibiting an explicit example showing that horizontal $W^{2,q}$ regularity of Calder\'on-Zygmund type for fully nonlinear subelliptic equations posed on the Heisenberg group cannot be in general expected in the range $q<Q$, $Q$ being the homogeneous dimension of the group.
\end{abstract}

\maketitle

\section{Introduction}
Recently, E. Braga and D. Moreira obtained an optimal $C^{1,\alpha}$ regularity result for semiconvex supersolutions to fully nonlinear uniformly elliptic equations with an unbounded source term $f\in L^q$, $q>n$, $n$ being the dimension of the ambient space, see Theorem 3.6 in \cite{BM}. This extended a previous a priori estimate on the modulus of continuity of the gradient for semiconvex supersolutions to linear uniformly elliptic equations proved by L. Caffarelli, R. Kohn, L. Nirenberg and J. Spruck \cite{CKNS}, cf. Corollary  of Lemma 2.2. An improvement of such a result up to $q=n$, along with the optimal interior regularity of the convex envelope of $L^n$-viscosity supersolutions to Pucci's extremal equations with unbounded coefficients, were investigated by E. Braga, A. Figalli and D. Moreira in \cite{BFM}.\\
Given a family $\X=\{X_1,...,X_m\}$ of $C^2$ vector fields of Carnot-type, the purpose of this note, inspired by the aforementioned works, is to establish a priori high-order estimates in second order Sobolev spaces (and in first-order H\"older spaces via Morrey-type embeddings as a byproduct) for solutions to degenerate fully nonlinear subelliptic inequalities of the form
\begin{equation}\label{eqintro}
G(x,(D^2_\X u)^*)\leq f(x)\text{ in }\Omega,
\end{equation}
where $(D^2_\X u)^*$ stands for the symmetrized horizontal Hessian along the vector fields $(D^2_\X u)^*=\frac{X_iX_ju+X_jX_iu}{2}$. Here, $G:\Omega\times\mathcal{S}_m\to\R$, $\Omega\subseteq\R^n$, $n\geq m$, $\mathcal{S}_m$ being the space of $m\times m$ symmetric matrices, is continuous and uniformly subelliptic in the sense that
\begin{equation}\label{uellintro}
\mathcal{M}^-_{\lambda,\Lambda}(X-Y)\leq G(x,X)-G(x,Y)\leq \mathcal{M}^+_{\lambda,\Lambda}(X-Y), X,Y\in \mathcal{S}_m,
\end{equation}
$\mathcal{M}^-_{\lambda,\Lambda},\mathcal{M}^+_{\lambda,\Lambda}$ being the Pucci's extremal operators, see \eqref{pucci+}-\eqref{pucci-} below. A longstanding open problem in the regularity theory of such equations or, more generally, subelliptic equations in nondivergence form with bounded measurable coefficients, is the validity of an analogue of the Krylov-Safonov Harnack inequality. This is mainly due to the lack of Aleksandrov-Bakel'man-Pucci (briefly ABP) maximum principles, see the introduction of \cite{DGN} for a thorough discussion. In this framework, it was only shown that ABP type estimates cannot be produced, at least in Heisenberg-type groups, when the right-hand side of the equation lies in $L^q$, $q<Q$, $Q$ being the homogeneous dimension \cite{DGN}, in accordance with the classical results of A. D. Aleksandrov and C. Pucci \cite{A,Pucci66}. Nonetheless, some invariant Harnack inequalities have been obtained under smallness conditions on the ratio among the ellipticity constants (this can be referred to Landis/Cordes-type conditions) for nondivergence structure elliptic problems posed on Carnot groups of Heisenberg-type, see \cite{GT,Tralli,AGT}, and also for parabolic Kolmogorov and kinetic operators with diffusion in nondivergence form, cf. \cite{AT}. For weighted ABP estimates and Harnack inequalities on Grushin geometries see \cite{M}. We also mention that a global Krylov-Safonov Harnack inequality was proved in \cite{C} for equations in nondivergence form on Riemannian manifolds under curvature lower-bounds. Note that the results in \cite{Tralli,AGT} immediately lead to an extension to the fully nonlinear setting, and also provide low-regularity in H\"older spaces using classical arguments, under suitable restrictions on the ellipticity constants. Still, a subelliptic version of the $C^{1,\alpha}$ estimate of Cordes-Nirenberg type seems not available, and to our knowledge no results appeared in the context of degenerate fully nonlinear subelliptic equations of second order in horizontal $C^{2,\alpha}_\X$ and $W^{2,q}_\X$ spaces, even for smooth solutions and smooth functionals. Indeed, the structure of the problem over horizontal Hessians prevents from the use of classical linearization arguments, based for example on the Bernstein method, see e.g. \cite{CC,CKNS,TruTAMS,RRO}. Typically, when $\X$ are the Euclidean vector fields, see \cite{RRO}, if $u$ smooth solves the model problem $F(D^2u)=0$, with $F\in C^1$ and uniformly elliptic in the sense of \cite{CC}, then $v=\partial_e u$, $e\in\R^n$ with $|e|=1$, solves a linear nondivergence structure equation of the form
\[
a_{ij}(x)\partial_{ij}v=0\ ,a_{ij}(x)=F_{ij}(D^2u).
\]
Then, the Krylov-Safonov H\"older regularity applies to $v$ (since no regularity properties on $a_{ij}$ are required) leading thus to $C^{1,\alpha}$ estimates under no assumptions on the nonlinearity other than the uniform ellipticity.\\
However, the non-commutative structure of sub-Riemannian geometries and the absence of the Krylov-Safonov theory do not allow to reproduce the previous approach. Similarly, the Bernstein method, based on a linearization argument, cannot be performed: this is easily seen by the recent Bochner-type formulas obtained in \cite{Gproc} for step-2 Carnot groups, which give rise to additional commutator terms.\\
In the classical regularity theory for fully nonlinear equations, though $C^{1,\alpha}$, $\alpha<1$, estimates are obtained under essentially no assumptions on $F=F(D^2u)$, $C^{1,1}$ estimates (and higher-order $C^{2,\alpha}$ estimates through the Evans-Krylov theorem) usually require concavity/convexity type assumptions on $F$. However, $C^{1,1}$ estimates were proved in \cite{Imbert} for strictly elliptic inequalities under the geometric assumption that the unknown function is convex: this was obtained through the properties of the convex envelope found in \cite{ALL}, without requiring neither convexity nor concavity on the fully nonlinear operator. The aforementioned works \cite{BM,BFM} lead to a further development, proving indeed a generalization of the results in \cite{CKNS,Imbert} for more general semiconvex semisolutions with unbounded coefficients in Lebesgue spaces. \\
In the subelliptic context, a generalized notion of convexity (and, thus, semiconvexity) along the vector fields, called $\X$-convexity, was discussed in \cite{BD}, see also \cite{DGNcag} for Carnot groups. The results in \cite{BD} say that $u\in\mathrm{USC}$ is convex along the fields of $\X$ if and only if $(D^2_\X \varphi(x))^*\geq0$ for all smooth $\varphi$ and $x\in\mathrm{argmax}(u-\varphi)$. \\
Therefore, having such a notion at our disposal and following \cite{BM}, we prove the following model result for horizontal Hessian inequalities structured over the fields of a Carnot group: if $u$ is $\X$-semiconvex according to \cite{BD} and a smooth solution to 
\[
G((D^2_\X u)^*)\leq f\text{ in }B_1
\]
with $f\in L^\infty(B_1)$, $G$ satisfying \eqref{uellintro} with $G(0)=0$ (with no other assumptions on $G=G(X)$), where $B_r$ is a metric ball of radius $r$, then the following interior a priori estimate holds:
\[
\|(D^2_\X u)^*\|_{L^q(B_\frac12)}\lesssim C(\|u\|_{L^q(B_1)}, \|f\|_{L^q(B_1)}),\ q\in(1,\infty).
\]
When $q$ is larger than a certain threshold involving the homogeneous dimension associated to the fields we get $C^{1,\alpha}$ estimates with respect to the corresponding sub-Riemannian distance. As stated in the previous results, throughout the paper we work with classical solutions, though all the estimates are independent of the smoothness and depend only on the structural constants and the integrability of the data. We briefly discuss some one-side second order estimates, that can be thus regarded as regularity estimates, in the viscosity sense, see Remark \ref{visc}.\\
 The prototype example to which our results apply is the (degenerate) Isaacs equation, i.e. a PDE of the form \[G((D^2_\X u)^*)=\sup_\beta\inf_\alpha \mathrm{Tr}(A_{\alpha\beta}(x)\ (D^2_\X u)^*)=\sup_\beta\inf_\alpha \mathrm{Tr}(A_{\alpha\beta}(x)\sigma^T(x)D^2u\sigma(x)),\] where $\sigma\in\R^{n\times m}$ is a (degenerate) matrix having the fields $X_i$ as columns, $\alpha,\beta$ belong to some control sets $A,B$ respectively, while $D^2u\in \mathcal{S}_n$ is the standard Hessian of the unknown function $u$. It is well-known that every nonlinear operator $G=G((D^2_\X u)^*)$ satisfying \eqref{uellintro} can be written in Isaacs form as follows, cf. \cite{CCjmpa}: by \eqref{uellintro} one has for $X,Y\in\mathcal{S}_m$ and $\mathcal{A}_{\lambda,\Lambda}=\{A\in\mathcal{S}_m,\lambda I_m\leq A\leq\Lambda I_m\}$,
\[
G(X)-G(Y)\leq \mathcal{M}^+_{\lambda,\Lambda}(X-Y)=\sup_{A\in\mathcal{A}_{\lambda,\Lambda}}\mathrm{Tr}(A(X-Y)).
\]
Since the equality is attained when $X=Y$, one obtains
\[
G(X)=\min_{Y\in\mathcal{S}_m}\max_{A\in\mathcal{A}_{\lambda,\Lambda}}\{ \mathrm{Tr}(AX)+G(Y)-\mathrm{Tr}(AY)\}.
\]
By taking then $X=(D^2_\X u)^*$ we have
\[
G((D^2_\X u)^*)=\min_{Y\in\mathcal{S}_m}\max_{A\in\mathcal{A}_{\lambda,\Lambda}}\{\mathrm{Tr}(A(D^2_\X u)^*)+G(Y)-\mathrm{Tr}(AY)\}.
\]
Note that the previous operator is degenerate elliptic if expressed in Euclidean coordinates. However, the previous announced results turn out to be new even for general linear subelliptic equations of the form $G((D^2_\X u)^*)=\mathrm{Tr}(A(x)D^2_\X u)^*)$, where $A$ satisfies $\lambda I_m\leq A\leq \Lambda I_m$ and has bounded measurable entries, and, finally, apply to any Carnot group.\\

Some comments on the results are now in order. Though the (geometric) a priori requirement on the solution could appear strong (and thus the consequent estimate conditional to the geometric bound), the procedure we are going to implement allows to conclude the estimate for semisolutions, and not only for solutions. Indeed, $W^{2,q}$ estimates are typically obtained for solutions to uniformly elliptic PDEs (which indicates a two-side analytic control on the problem, since, roughly speaking, $G=f$ with $|f|\leq C$ means $-C\leq G\leq C$). Indeed, classical results for the Poisson equation show that if $u$ solves $\Delta u=f\in L^\infty$, i.e. $-C_f\leq \Delta u\leq C_f$ for some $C_f>0$, then $u\in C^{1,\alpha}\cap W^{2,q}$ and the following estimates hold
\[
\|u\|_{W^{2,q}(B_\frac12)}\leq C(n)(\|u\|_{L^\infty(B_1)}+C_f),
\]
\[
\|u\|_{C^{1,\alpha}(B_\frac12)}\leq C(n)(\|u\|_{L^\infty(B_1)}+C_f).
\]
Motivated by these results, one can ask whether it is possible to expect some (high) level of regularity of $u$ when one has only an analytic control from above, i.e. $\Delta u\leq C$, based on the regularity properties of the right-hand side and a geometric control from below involving second derivatives, such as a convexity-type assumption. \\
Besides, the analysis of the regularity properties of semisolutions to fully nonlinear degenerate equations satisfying geometric conditions like convexity is motivated by the understanding of a regularity theory for such degenerate PDEs. Indeed, a crucial step in the proof of the ABP maximum principle for equations over Euclidean vector fields in \cite{C91} is to show that the convex envelope of supersolutions to $\mathcal{M}^-_{\lambda,\Lambda}(D^2u)\leq f$ is $C^{1,1}_{\mathrm{loc}}$ when $f\in L^\infty$, cf. Chapter 3 of \cite{CC}. Thus, a tightly related question is the study of the regularity of supersolutions satisfying convexity constraints and linear or nonlinear partial differential inequalities in nondivergence form. In this direction, a study of the horizontal convex envelope in the Heisenberg group, along with its application to the study of horizontal convexity properties of solutions to fully nonlinear equations was performed in \cite{LiuZhou}. This analysis has its roots in the earlier work by O. Alvarez, J.-M. Lasry and P.-L. Lions \cite{ALL}. Based on this study, in the course of Section \ref{sec;conv} we provide some sufficient conditions to obtain convexity estimates on solutions, providing an instance of the full result unconditional to the geometric bound. Another paper dealing with such convexity preserving properties for nondivergence structure uniformly parabolic PDEs is \cite{LM}. \\
Other than the previous motivations, this seems the first (a priori) regularity result concerning fully nonlinear subelliptic equations in higher-order spaces. Few results are available in the literature of nondivergence subelliptic equations: the works \cite{Tralli,AGT} allow to deduce a priori estimates at the level of $C^{\alpha}$ spaces for subelliptic Hessian equations with bounded right-hand side. Indeed, in the smooth setting one can write
\[
G((D^2_\X u)^*)=\int_0^1 G_{ij}(t(D^2_\X u)^*))\,dt\ (D^2_\X u)_{ij}^*+G(0)=\mathrm{Tr}(A(x)(D^2_\X u)^*)+G(0),
\]
so
\[
G((D^2_\X u)^*)=0\iff \mathrm{Tr}(A(x)(D^2_\X u)^*)=-G(0),
\]
and apply e.g. the results from \cite{Tralli,AGT} valid when the ellipticity constants are sufficiently close to each other.\\
We also prove, as a further step towards a possible development of a regularity theory for fully nonlinear subelliptic equations, that $W^{2,q}_\X$ estimates for the uniformly subelliptic equation
\[
G((D^2_\X u)^*)= f\in L^q
\]
fail when $q<Q$, at least in the Heisenberg group, $Q$ being its homogeneous dimension. This shows that the best possible integrability order to get such estimates is $q=Q$.  This is done adapting a counterexample proposed by L. Caffarelli in \cite{C91}. We emphasize that in this sub-Riemannian context there is a dimensional discrepancy between the dimension of the horizontal layer, say $m$, and the homogeneous dimension $Q$ (and even the topological dimension $n$), so it is unclear the right integrability order to deduce such an estimate. This shares some similarities with a result obtained in \cite{DGN} for linear subelliptic equations in nondivergence form with bounded measurable coefficients in H-type groups, concerning the validity of the uniqueness of solutions in the Sobolev space $W^{2,Q}_\X$. The example in \cite{DGN} also implies that ABP-type estimates cannot hold when the right-hand of the equation does not belong to $L^q$, $q\geq Q$. 
\par\medskip
The paper is organized as follows. Section \ref{sec;prel} gives some preliminaries on sub-Riemannian geometries and convexity along vector fields. Section \ref{sec;ell} is devoted to the proof of the a priori estimate for elliptic inequalities. Section \ref{sec;count} ends the paper with the failure of second order estimates in $L^q$ spaces when $q<Q$.

%
\section{Preliminaries on subelliptic structures}\label{sec;prel}
We denote by $\X=\{X_1(x),...,X_m(x)\}$ a family of $C^2$-vector fields in $\R^n$ and $\Omega\subset\R^n$ an open and connected set. We recall that the Carnot-Carath\'eodory (briefly CC), or sub-Riemannian distance $d$ on $\R^n$ associated to the vector fields of the family $\X$, is the length of the shortest horizontal curves connecting two-points, and is denoted by $d$. The pair $(\R^n,d)$ is a sub-Riemannian geometry if $d(x,y)<+\infty$ for all $x,y\in\R^n$ and the vector fields $\X$ are the generators. Important examples of such spaces are Carnot groups.\\

A stratified group (or Carnot group)
 \cite[Definition 2.2.3]{BLU} is a 
 connected and simply connected Lie 
group whose Lie algebra 
$\mathcal{G}$ 
admits a stratification 
$\mathcal{G}=\oplus_{i=1}^r V_i$, where the layers $V_i$ satisfy the relations $[V_i,V_{i-1}]=V_i$ for 
$2\leq i\leq r$ and $[V_1,V_r]=0$, with 
$[V,W]:=\mathrm{span}\{[v,w]: v\in V, w\in W\}$.
Here, $r$ is called the step of the group, 
and we set $m=n_1=\mathrm{dim}(V_1)$, which stands for the dimension of the horizontal layer, 
$n_i=\mathrm{dim}(V_i)$, 
$2\leq i\leq r$. A stratified group can 
be identified with a homogeneous 
Carnot group up to an isomorphism 
 by means of \cite[Section 2.2.3]{BLU}. 
A  
 homogeneous Carnot group
$\mathbb{G}$ 
(see e.g. \cite[Definition 1.4.1]{BLU}) 
can be identified with 
$\R^n=\R^{n_1}\times...\times\R^{n_r}$ 
($n=\sum_{i=1}^rn_i$) endowed with a group law 
$\star$
 if for  any 
$\lambda>0$ the dilation 
$\delta_\lambda:\R^n\to\R^n$ of the form 
$\delta_\lambda(x)=(\lambda x^{(1)},...,{\lambda^{r}}x^{(r)})$
is an automorphism of the group, where $x=(x^{(1)},....,x^{(r)})$, 
$x^{(i)}\in\R^{n_i}$.
Then the vector fields in
$\R^n$ of the family $\X=\{X_1,...,X_m\}$ generate the 
homogeneous Carnot group $(\R^n,\star,\delta_\lambda)$
if they are left-invariant on $\G$ and such that 
$X_j(0)=\partial_{x_j}|_0$ for $j=1,...,n_1$ span 
$\R^n$ at every point $x\in\R^n$.
In this case we say that $\G$ has step $r$ and $m=n_1$ generators. We also denote by 
$\Delta_{\X}=\sum_{i=1}^mX_i^2$ the operator
sum of squares of vector fields, usually known as
sub-Laplacian on the Carnot group $\mathbb{G}$. We also denote with 
$$
Q:=\sum_{i=1}^rin_i{ =\sum_{i=1}^ri \mathrm{dim}(V_i)}
$$ 
the homogeneous dimension of the group, with 
$$
D_{\X}u:= (X_1u,...,X_mu)
$$ 
the horizontal gradient along the vector fields, and with 
$D^2_{\X}u=X_iX_ju$, $i,j=1,...,m$, the 
horizontal Hessian built over the frame $\X$. Finally, $(D^2_{\X}u)^*=\frac{X_iX_ju+X_jX_iu}{2}\in\mathcal{S}_m$ is the symmetrized horizontal Hessian.\\ We point out that it was proved in \cite{BLU} that the fields of a Carnot group can be written as
\[
X_j=\partial_{x_j}+\sum_{i=m+1}^{n}b_{ij}(x)\partial_{x_i}, j=1,....,m.
\]
where $b_{ij}(x)=\sigma_{ij}(x_1,...,x_{i-1})$ are homogeneous polynomials of degree less than or equal to $n-m$. A model example that we will analyze in Section \ref{sec;count} is the Heisenberg group in $\mathbb{H}^d\simeq\R^{2d+1}$, $n=2d+1$, $m=2d$, $Q=2d+2$, where the fields of the horizontal layer are
\[
X_i=\partial_{x_i}+2x_{i+d}\partial_{2d+1}\ ,X_{i+d}=\partial_{x_i}-2x_{i}\partial_{2d+1},\ i=1,...,d.
\]
We denote by $B_r(x)=\{y\in\mathbb{G}:d(x,y)<r\}$ the metric ball of radius $r$. \\
We recall that for $u$ smooth, the convexity of $u$ along the fields of the family $\X$ is equivalent to the positive semidefiniteness of the symmetrized horizontal Hessian $(D^2_\X u)^*$. This has been extended to the viscosity setting in \cite{BD} to which we refer for further details. 
\section{Interior a priori estimates for elliptic inequalities under $\X$-semiconvexity conditions}\label{sec;ell}
\subsection{High-order local estimates under geometric conditions}
Let $\X$ be a family of vector fields of Carnot-type, and let $X_1,...,X_m$ be the generators of the group, where $m$ is the dimension of the horizontal layer. We assume that $G:\Omega\times\R\times\R^m\times\mathcal{S}_m\to\R$, $\Omega\subseteq\R^n$, $n\geq m$, is continuous and uniformly subelliptic, i.e.
\begin{equation}\label{uell}
\mathcal{M}^-_{\lambda,\Lambda}(X-Y)\leq G(x,r,p,X)-G(x,r,p,Y)\leq \mathcal{M}^+_{\lambda,\Lambda}(X-Y)
\end{equation}
for $X,Y\in\mathcal{S}_m$ and $(x,r,p)\in \Omega\times\R\times\R^m$, where
\begin{equation}\label{pucci+}
\mathcal{M}^+_{\lambda,\Lambda}(M)=\Lambda\sum_{e_k>0}e_k+\lambda \sum_{e_k<0}e_k=\sup\{\mathrm{Tr}(AM),\lambda I_m\leq A\leq AI_m\}
\end{equation}
\begin{equation}\label{pucci-}
\mathcal{M}^-_{\lambda,\Lambda}(M)=\lambda \sum_{e_k>0}e_k+\Lambda\sum_{e_k<0}e_k=\inf\{\mathrm{Tr}(AM),\lambda I_m\leq A\leq AI_m\}
\end{equation}
and $e_k$ denotes the $k$-th eigenvalue of $M\in\mathcal{S}_m$.\\
We prove the following interior a priori estimate in the horizontal spaces $W^{2,q}_\X(\Omega)=\{u\in L^q(\Omega):X_Iu\in L^q(\Omega)\text{ for any }|I|\leq 2\}$ and $C^{1,\alpha}_\X$, the classical first-order H\"older space with respect to the sub-Riemannian distance. Note that even if $u\in C^2$ and $f$ is bounded, bounds will depend only on the summability of the data, so this can be regarded as an a priori estimate under the one-side geometric condition that $u$ is semiconvex along the fields of the family $\X$.
\begin{thm}\label{main1} 
Let $u\in C^2(B_1)\cap L^q(B_1)$, $q\in(1,\infty)$, be a classical and $\X$-semiconvex solution (with constant $4C$) to the inequality
\begin{equation}\label{eqmain1}
G(x,u,D_\X u,(D^2_\X u)^*)\leq f\text{ in }B_1
\end{equation}
with $f\in L^\infty(B_1)$, $G$ satisfying \eqref{uell} with $G(x,u,D_\X u,0)$ bounded (independently of $u$), i.e. \[|G(x,u,D_\X u,0)|\leq M,\] with $M$ independent of $u$. Then, there exist a constant $\tilde C$ depending on $C,m,\lambda,\Lambda,M,q$ and a universal constant $K_1=K_1(m,q,\lambda,\Lambda)$  such that
\[
\|u\|_{W^{2,q}_\X(B_\frac12)}\leq K_1(\tilde C+\|u\|_{L^q(B_1)}+\|f\|_{L^q(B_1)}).
\]
If, in addition, $q>Q$ and $u\in L^\infty(B_1)$, we have the estimate
\[
\|u\|_{C^{1,1-\frac{Q}{q}}_\X(B_\frac12)}\leq K_2(\tilde C+\|u\|_{L^\infty(B_1)}+\|f\|_{L^q(B_1)}).
\]
Here, $K_1$ and $K_2$ are universal constants depending on $m,q,\lambda,\Lambda$, while $K_2$ depends also on the constant of the corresponding Sobolev embedding.
\end{thm}
\begin{proof}
We follow \cite[Remark 1.2]{BM}. Since $u$ is smooth and $\X$-semiconvex with constant $4C$, then $v=u+2C\sum_{i=1}^mx_i^2=u+z$ is $\X$-convex. Therefore, recalling that, cf. \cite{CC},
\begin{equation}\label{pucci1}
\mathcal{M}^-_{\lambda,\Lambda}(X+Y)\geq \mathcal{M}^-_{\lambda,\Lambda}(X)+\mathcal{M}^-_{\lambda,\Lambda}(Y)
\end{equation}
and
\begin{equation}\label{pucci2}
\mathcal{M}^-_{\lambda,\Lambda}(-X)=-\mathcal{M}^+_{\lambda,\Lambda}(X)
\end{equation}
we have, using first \eqref{uell} with $r=v-z, p=D_\X(v-z),X=(D^2_\X (v-z))^*$ and $Y=0$, then \eqref{pucci1}-\eqref{pucci2} and also that $v$ is $\X$-convex
\begin{align*}
f\geq G(x,v-z,D_\X(v-z),(D^2_\X (v-z))^*)&\geq \mathcal{M}^-_{\lambda,\Lambda}((D^2_\X (v-z))^*)+G(x,v-z,D_\X(v-z),0)\\
&\geq \mathcal{M}^-_{\lambda,\Lambda}((D^2_\X v)^*)+\mathcal{M}^-_{\lambda,\Lambda}((D^2_\X (-z))^*)+G(x,u,D_\X u,0)\\
&=\mathcal{M}^-_{\lambda,\Lambda}((D^2_\X v)^*)-\mathcal{M}^+_{\lambda,\Lambda}((D^2_\X z)^*)+G(x,u,D_\X u,0)\\
&\geq \lambda \mathrm{Tr}((D^2_\X v)^*)-4mC\Lambda+G(x,u,D_\X u,0)\\
&=\lambda\left(\sum_{i=1}^mX_i^2 u+4mC\right) -4mC\Lambda+G(x,u,D_\X u,0)\\
&\geq \lambda\sum_{i=1}^mX_i^2 u-4mC(\Lambda-\lambda)-M.
\end{align*}
Therefore, the last estimate together with the $\X$-semiconvexity give the pointwise bound
\begin{equation}\label{point}
-4Cm\leq \sum_{i=1}^mX_i^2 u(x)\leq \lambda^{-1}(f(x)+4mC(\Lambda-\lambda)+M), x\in B_1.
\end{equation}
Therefore, using that $u$ is $\X$-semiconvex, we get for a constant $K=K(m)>0$ and $i,j=1,...,m$
\begin{equation}\label{mixed}
|X_iX_ju|\leq \sum_{i,j=1}^m|X_iX_ju|\leq K(\Delta_\X u+4Cm).
\end{equation}
Indeed, denoting by $\|X\|_{\mathrm{spec}}=\max\{e_i(X)\}$ the spectral norm of $X\in\mathcal{S}_m$ and by $\|X\|_{1}=\sum_{i,j=1}^m|x_{ij}|$, since $X\in\mathcal{S}_m$, we know that for a constant $L(m)>0$ it holds \[\|X\|_1\leq L(m)\|X\|_{\mathrm{spec}}.\] This implies, since $v$ is $\X$-convex, the following inequality
\[
|X_iX_jv|\leq L(m)\|(D^2_\X v)^*\|_{\mathrm{spec}}\leq  L(m)\mathrm{Tr}(X_iX_jv)= L(m)\Delta_\X v=L(m)(\Delta_\X u +4Cm).
\]
We conclude by noting that \[|X_iX_jv|=|X_iX_ju+4C\delta_{ij}|\geq |X_iX_ju|.\] \\
By interpolation estimates, see e.g. the proof of Proposition 2.9 in \cite{BBB}, for every $\delta>0$ there exists a constant $c_q>0$ such that
\begin{equation}\label{interp}
\|D_\X u\|_{L^q(B_\frac12)}\leq \delta\|\Delta_\X u\|_{L^q(B_1)}+\frac{c_q}{\delta}\|u\|_{L^q(B_1)}.
\end{equation}
Recalling the definition of the $W^{2,q}_\X$ norm and combining \eqref{point}-\eqref{mixed}-\eqref{interp}, we obtain for a universal constant $K_1=K_1(m,q,\lambda,\Lambda)$ and a constant $\tilde C=\tilde C(\lambda,\Lambda,m,C,q)$, both independent of $u$,
\[
\|u\|_{W^{2,q}_\X(B_\frac12)}\leq K_1(\tilde C+\|u\|_{L^q(B_1)}+\|f\|_{L^q(B_1)}+M).
\]
If $q>Q$ and $u$ is bounded, by the horizontal Morrey embedding, see e.g. \cite{Folland}, we have 
\[
\|u\|_{C^{1,1-\frac{Q}{q}}_\X(B_\frac12)}\leq K_2(\tilde C+\|u\|_{L^\infty(B_1)}+\|f\|_{L^q(B_1)})
\]
for a different universal constant $K_2$ depending on $K_1$ and $C_s$, the latter being the constant of the embedding $W^{2,q}_\X\hookrightarrow C^{1,\alpha}_\X$.
\end{proof}

Some remarks are in order:
\begin{rem}\label{visc}
The estimate in the right-hand side of \eqref{point} can be made rigorous in the weak setting of viscosity solutions by means of the transitivity of the viscosity inequalities. More precisely, adapting Lemma 2.12 in \cite{CC} in the degenerate setting of H\"ormander vector fields, one can prove that if $u\in\mathrm{LSC}(B_1)$ is a viscosity solution to \eqref{eqmain1} with $f\in C(B_1)$, then there exists a constant $K$ depending on $\sup_{B_1}f$, $\lambda,\Lambda, |G(x,u,D_\X u,0)|,m,C$ such that
\begin{equation*}
\sum_{i=1}^mX_{i}X_j \varphi(x)\leq K
\end{equation*}
provided $\varphi\in C^2$ and $u-\varphi$ has a local minimum at $x$.
\end{rem}
\begin{rem}\label{strict}
The estimates in Theorem \ref{main1} can be obtained for the smaller class of $\X$-convex supersolutions, but under the weaker condition of strict (sub)ellipticity
\[
\lambda\mathrm{Tr}(X-Y)+G(x,r,p,Y)\leq G(x,r,p,X)
\]
or the one-side condition
\[
G(x,r,p,X)\geq G(x,r,p,0)+\mathcal{M}^-_{\lambda,\Lambda}(X).
\]
For instance, the linear operator $G(x,X)=\mathrm{Tr}(A(x)X)$ satisfies the first condition when $A\geq\lambda I_m$. In this case one can get $C^{1,1}_\X$ estimates along vector fields, as it is done in Theorems 3 and 4 of \cite{Imbert}. It is enough to prove that $u$ is a solution to the inequality
\[
\lambda \Delta_\X u+G(x,u,D_\X u,0)\leq0\text{ in }\R^n.
\]
To see the validity of $C^{1,1}$ estimates when $u$ is smooth, it is sufficient to exploit the strict ellipticity, together with the fact that $(D^2_\X u)^*\geq0$ holds in the classical sense and that, for $A\geq0$, $A\in\mathcal{S}_m$, we have the inequality $A\leq C(m)\mathrm{Tr}(A)I_m$. The following estimate then follows for $\X$-Lipschitz solutions by exploiting the equation
\[
\|(D^2_\X u)^*\|\leq \frac{C(m)}{\lambda}|G(x,u,D_\X u,0)|.
\]
A similar estimate holds when $u$ is $\X$-semiconvex. 
\end{rem}
\begin{rem}
It is worth noting that one cannot expect the validity of the above a priori estimates when $q<Q$. An example is the following, cf. Example 3.4 in \cite{BFM}: consider in the Heisenberg group $\mathbb{H}^d\simeq \R^{2d+1}$ the function $u(x)=\rho(x)-1$ in $B_1$, $\rho(x)=(|x_H|^4+x_V^2)^\frac14$, $x_H=(x_1,...,x_{2d})$ and $x_V=x_{2d+1}$. Then, $u$ is $\X$-convex for $|x_H|\neq0$ by Theorem 6.6 of \cite{DGNcag} and solves $\Delta_\X u=\frac{Q-1}{\rho}|D_\X\rho|^2=:f(x)$. Then, $f\in L^q(B_1)$ if and only if $q<Q$, but $u$ is not differentiable at points where $(x_H,x_V)=(0,0)$ since for $i=1,...,d$ we have
\[
X_i\rho=x_i\frac{|D_\X\rho|^2}{\rho}+\frac{x_{i+d}x_{2d+1}}{\rho^3},\ \ \ X_{i+d}\rho=x_{i+d}\frac{|D_\X\rho|^2}{\rho}-\frac{x_{i}x_{2d+1}}{\rho^3}.
\]
\end{rem}

\subsection{Convexity and semiconvexity estimates for fully nonlinear equations}\label{sec;conv}
In the theory of elliptic and parabolic equations, geometric one-side second order estimates play a crucial role. To our knowledge there are three methods allowing to prove convexity and semiconvexity properties of solutions for elliptic and parabolic equations in the Euclidean setting. The first one dates back to \cite{K}, where it was developed the so-called concavity maximum principle, later adapted to the weak setting of viscosity solutions \cite{Giga}. Another method, based again on the comparison principle, was introduced in \cite{ALL} and it is based on the fact that the convex envelope of a viscosity supersolution remains a supersolution to the starting equation. The more recent one, recently introduced in \cite{EvansARMA10} and based on integral methods, allows to prove certain concavity-type preserving estimates for nonlinear viscous problems under appropriate assumptions on the nonlinearity. This also applies to fully nonlinear diffusions under appropriate regularity assumptions on the nonlinear terms.\\
The extensions of these properties to a degenerate setting such as those of Carnot groups is by no means immediate. The work \cite{Manfredi} shows by tripling the number of variables \`a la Ishii-Lions a right-invariant convexity property for evolution equations in nondivergence form. The more recent paper \cite{LiuZhou} studies the (left) horizontal convexity preserving property and proved the following result for fully nonlinear subelliptic equations posed on the first Heisenberg group in $\R^3$: if $G(x,r,p,X)$ is proper, concave in all the variables and symmetric with respect to $(x,p)$ (see assumptions (A3) and (A4) in \cite{LiuZhou}), continuous solutions to fully nonlinear subelliptic equations on the Heisenberg group are horizontally convex provided that the comparison principle holds for coercive solutions, cf. \cite[Theorem 5.7]{LiuZhou}. By Lemma 4.1 in \cite{BD} we would conclude that $u$ is $\X$-convex along the fields of the Heisenberg group. Consequently, we have the following result:
\begin{thm}
Suppose that $G$ is uniformly subelliptic, proper, concave in all the variables and symmetric with respect to $x,p$. Then, coercive solutions to $G(x,u,D_\X u,(D^2_\X u)^*)=0$ posed on the first Heisenberg group $\mathbb{H}$ satisfy the a priori estimates of Remark \ref{strict}.
\end{thm}
We remark that convexity and semiconvexity type properties hold for the more restrictive class of solutions and typically require additional conditions on the nonlinearity (such as concavity), as it happens when the equation is driven by the sole (sub-)Laplacian in the Euclidean setting.

\section{Impossibility of $W^{2,q}_\X$ estimates for solutions fully nonlinear subelliptic equations when $q<Q$}\label{sec;count}
In this section we turn to the smaller class of solutions to elliptic equations, without imposing geometric conditions on them. We prove the following
\begin{thm}\label{cexcz}
There exists a fully nonlinear subelliptic operator modeled on the horizontal Hessian of the Heisenberg vector fields, i.e. $G(X):\mathcal{S}_m\to\R$, such that $W^{2,q}_\X$ a priori estimates for the equation $G((D^2_\X u)^*)=f(x)$ fail when $q<Q$.
\end{thm}
To prove the result we premise the following algebraic lemma, referring for its proof to \cite{CT}.
\begin{lemma}\label{hesH}
Let $\X$ be the Heisenberg vector fields in $\R^{2d+1}$ and $\rho$ be the homogeneous norm defined as
\begin{equation}\label{norm}
\rho(x)=\left(\left(\sum_{i=1}^{2d}x_i^2\right)^2+x_{2d+1}^2\right)^\frac14.
\end{equation}
Then, for $|x_H|\neq0$, being $x_H=(x_1,...,x_{2d})$,
\[
D_{\X}\rho = \frac{\eta}{\rho^3} , \qquad |D_{\X}\rho|^2=\frac{|x_H|^2}{\rho^2}\leq 1\, ,
\]
where $\eta\in\R^{2d}$  
is defined by
\begin{equation*}
\label{eta}
\eta_i:=x_i|x_H|^2+x_{i+d}x_{2d+1}\ , \qquad
\eta_{i+d}:=x_{i+d}|x_H|^2-x_{i}x_{2d+1} .
\end{equation*}
for $i=1,...,d$. In addition
\[
(D^2_{\X}\rho)^*=-\frac{3}{\rho}D_{\X}\rho\otimes D_{\X}\rho+\frac{1}{\rho}|D_{\X}\rho|^2I_{2d}+\frac{2}{\rho^3}\begin{pmatrix}B & C\\-C & B\end{pmatrix}\ ,
\]
where the matrices $B=(b_{ij})\in\R^{d\times d}$ and $C=(c_{ij})\in \R^{d\times d}$ are defined as follows
\[
b_{ij}:=x_ix_j+x_{d+i}x_{d+j}\ ,\qquad c_{ij}:=x_ix_{d+j}-x_jx_{d+i}
\]
for $i,j=1,...,d$ and satisfy $B=B^T$ and $C^T=-C$.
Finally, for a radial function $\psi=\psi(\rho)$ we have
\[
(D^2_{\X}\psi)^*(\rho)=\frac{\psi'(\rho)|D_{\X}\rho|^2}{\rho}I_{2d}+2\frac{\psi'(\rho)}{\rho^3}\begin{pmatrix}B & C\\-C & B\end{pmatrix}+\left(\psi''(\rho)-3\frac{\psi'(\rho)}{\rho}\right)D_{\X}\rho\otimes D_{\X}\rho\ ,
\]
and its eigenvalues are $\psi''(\rho)|D_{\X}\rho|^2, 3\psi'(\rho)\frac{|D_{\X}\rho|^2}{\rho}$, which are simple, and $\psi'(\rho)\frac{|D_{\X}\rho|^2}{\rho}$ with multiplicity $2d-2$.
\end{lemma}

\begin{proof}[Proof of Theorem \ref{cexcz}]
This is inspired from the work by L. Caffarelli \cite{C91}. Consider, for $q<Q$, $\varepsilon>0$, $0<\alpha<1$, the family of functions $u_{\varepsilon,\alpha}(x)=\psi(\rho(x))$ defined as 
\[
u_{\varepsilon,\alpha}(x)=
\begin{cases}
1-(\rho(x))^\alpha&\text{ for }\rho>\varepsilon\\
1-\alpha\varepsilon^{\alpha-2}(\rho(x))^2-(1-\alpha)\varepsilon^\alpha&\text{ for }\rho\leq\varepsilon,
\end{cases}
\]
$\rho$ being the homogeneous norm of the Heisenberg group defined in \eqref{norm}. Then, we have
\[
\psi_{\rho\rho}=
\begin{cases}
\alpha(1-\alpha)\rho^{\alpha-2}&\text{ for }\rho>\varepsilon\\
-2\alpha\varepsilon^{\alpha-2}&\text{ for }\rho<\varepsilon.
\end{cases}
\]
and
\[
\frac{1}{\rho}\psi_{\rho}=
\begin{cases}
-\alpha\rho^{\alpha-2}&\text{ for }\rho>\varepsilon\\
-2\alpha\varepsilon^{\alpha-2}&\text{ for }\rho<\varepsilon.
\end{cases}
\]
Then, define, for $\alpha<1$ fixed, the constants $\Lambda=\frac{1}{1-\alpha}$ and $\lambda=\frac{1}{Q-1}$ (note that $\Lambda>\lambda$) and consider the fully nonlinear operator $\mathcal{M}^+_{\lambda,\Lambda}(X)$. Note also that it is a convex degenerate operator. We apply it to the symmetrized horizontal Hessian of the family $u_{\varepsilon,\alpha}$, i.e. $X=(D^2_\X u_{\varepsilon,\alpha})^*$.  Then, by Lemma \ref{hesH}, the eigenvalues of $(D^2_\X u_{\varepsilon,\alpha})^*$ (where $u_{\varepsilon,\alpha}$ is seen as a radial function of $\rho$) are $|D_\X \rho|^2\psi_{\rho\rho}$, $3|D_\X \rho|^2\frac{1}{\rho}\psi_{\rho}$ which are simple, and $|D_\X \rho|^2\frac{1}{\rho}\psi_{\rho}$ with multiplicity $2d-2$. Therefore, being $B_\varepsilon$ the homogeneous metric ball of radius $\varepsilon$, we have for $\rho<\varepsilon$ and $|x_H|\neq0$
\[
\mathcal{M}^+_{\lambda,\Lambda}((D^2_\X u_{\varepsilon,\alpha})^*)=\lambda|D_\X\rho|^2(-2\alpha\varepsilon^{\alpha-2}-2\alpha\varepsilon^{\alpha-2}(Q-1))=-C_1(\alpha,Q)|D_\X\rho|^2\varepsilon^{\alpha-2}\chi_{B_\varepsilon}=:f_{\varepsilon,\alpha}(x),
\]
where $C_1(\alpha,Q)=\frac{2\alpha Q}{Q-1}$ and $\chi_{B_\varepsilon}$ is the indicator function of $B_\varepsilon$. Instead, when $\rho>\varepsilon$ we conclude
\[
\mathcal{M}^+_{\lambda,\Lambda}((D^2_\X u_{\varepsilon,\alpha})^*)=\Lambda|D_\X\rho|^2\alpha(1-\alpha)\rho^{\alpha-2}-\lambda(Q-1)\alpha|D_\X\rho|^2\rho^{\alpha-2}=0.
\]
When $|x_H|=0$ we have $\mathcal{M}^+_{\lambda,\Lambda}((D^2_\X u_{\varepsilon,\alpha})^*)=0$. Therefore, using that $|D_\X\rho|\leq1$ and applying Proposition 5.4.4 of \cite{BLU} (denoting by $\omega_q=|B_1|$), it follows that for $\alpha<1$ fixed
\begin{align*}
\|f_{\varepsilon,\alpha}\|_{L^q}^q&=\int_{B_\varepsilon}C_1^q(\alpha,Q)|D_\X\rho|^{2q}\varepsilon^{(\alpha-2)q}\,dx\\
&\leq C_1^q(\alpha,Q)\varepsilon^{(\alpha-2)q}\int_{B_\varepsilon}\,dx= C_1^q(\alpha,Q)\varepsilon^{(\alpha-2)q}Q\omega_q\int_0^\varepsilon s^{Q-1}\,ds\\
&=C_2(\alpha,q,Q)\varepsilon^{(\alpha-2)q+Q}\to0\text{ as }\varepsilon\to0
\end{align*}
for any
\[
q<\frac{Q}{2-\alpha}.
\]
Therefore, for $|x_H|\neq0$, $\|f_{\varepsilon,\alpha}\|_{q}\leq\overline{K}$ remains bounded for all $\varepsilon\in(0,1]$, whereas for such value of $q$ the norm $\|u_{\eps,\alpha}\|_{W^{2,q}_\X(\R^n)}$ blows-up because of the definition of $\psi_{\rho\rho}$ for $\rho>\eps$, showing that second order estimates in $L^q$ cannot be achieved in the range $q<Q$.
\end{proof}

\begin{rem}
The same counterexample can be built on H-type groups using the computations of $(D^2_\X \rho)^*$ carried out in \cite{Tralli}. 
\end{rem}

\end{document}